\documentclass[12pt]{amsart}   %%%%%%%Please do not change
\usepackage{amscd,amssymb,amsmath,graphics}                  %%%%%%%
\usepackage[matrix,arrow,curve]{xy}
\newtheorem{theorem}{Theorem}[section]
\newtheorem{lemma}[theorem]{Lemma}
\newtheorem{fact}[theorem]{Fact}
\newtheorem{proposition}[theorem]{Proposition}
\newtheorem{question}[theorem]{Question}
\newtheorem{corollary}[theorem]{Corollary}
\newtheorem{example}[theorem]{Example}
\newtheorem{problem}[theorem]{Problem}
\theoremstyle{definition}
\newtheorem{definition}[theorem]{Definition}
\theoremstyle{remark}

\numberwithin{equation}{section}

\begin{document}

\vspace{0.5in}

%%%%%%%%%%%%%%%%
\newcommand{\abs}[1]{\lvert#1\rvert}
\def\norm#1{\left\Vert#1\right\Vert}
\def\Q {{\Bbb Q}}
\def\I {{\Bbb I}}
\def\C {{\Bbb C}}
\def\N{{\Bbb N}}
\def\R{{\Bbb R}}
\def\di{{\mathrm{di}}}
\def\Z {{\Bbb Z}}
\def\U{{\Bbb U}}
\def\F{{\mathrm{E}}}
\def\Un{{\mathcal{U}}}
\def\Is{{\mathrm{Is}}\,}
\def\Aut{{\mathrm {Aut}}\,}
\def\supp{{\mathrm {supp}}\,}
\def\Homeo{{\mathrm{Homeo}}\,}
\def\gr{{\underline{\Box}}}
\def\diam{{\mathrm{diam}}\,}
\def\d{{\mathrm{dist}}}
\def\H{{\mathcal H}}
\def\me{{\mathrm{me}}}

\def\a{\alpha}
\def\d{\delta}
\def\D{\Delta}
\def\g{\gamma}
\def\s{\sigma}
\def\Si{\Sigma}
\def\implies{\Rightarrow}
\def\o{\omega}
\def\O{\Omega}
\def\G{\Gamma}

\def\sB{{\mathcal B}}
\def\sC{{\mathcal C}}
\def\sE{{\mathcal E}}
\def\sF{{\mathcal F}}
\def\sG{{\mathcal G}}
\def\sH{{\mathcal H}}
\def\sJ{{\mathcal J}}
\def\sK{{\mathcal K}}
\def\sL{{\mathcal L}}
\def\sM{{\mathcal M}}
\def\sN{{\mathcal N}}
\def\sO{{\mathcal O}}
\def\sP{{\mathcal P}}
\def\sR{{\mathcal R}}
\def\sS{{\mathcal S}}
\def\sT{{\mathcal T}}
\def\sU{{\mathcal U}}
\def\sV{{\mathcal V}}

\def\sbs{\subset}
\def\rar{\rightarrow}
\def\e{\epsilon}

\def\ti{\times}
\def\obr{^{-1}}
\def\stm{\setminus}
\def\newline{\hfill\break}

\def\Exp{{\mathrm{Exp}}\,}
\def\Iso{{\mathrm{Iso}}\,}
\def\Sym{{\mathrm{Sym}}\,}

%=================================================================

\title[Is the free locally convex space $L(X)$ nuclear?]{Is the free locally convex space $L(X)$ nuclear?}
\date{\today}
\author{Arkady Leiderman} 

\address{Department of Mathematics, Ben-Gurion University of the Negev, Beer Sheva, P.O.B. 653, Israel}

\email{arkady@math.bgu.ac.il}

\author{Vladimir Uspenskij}

\address{321 Morton Hall, Department of Mathematics, Ohio University, Athens, Ohio 45701, USA}

\email{uspenski@ohio.edu}

\dedicatory{Dedicated to Mar{\'\i}a Jes\'us Chasco on the occasion of her birthday}  

\begin{abstract}
 Given a class $\sP$ of Banach spaces, a locally convex space (LCS) $E$
is called {\em multi-$\sP$} if $E$ can be isomorphically embedded into a product of spaces that belong to $\sP$.
We investigate the question whether the free locally convex space $L(X)$ is strongly nuclear, nuclear, Schwartz, 
multi-Hilbert or multi-reflexive. 

If $X$ is a Tychonoff space containing an infinite compact subset
then, as it follows from the results of \cite{Aus}, $L(X)$ is not nuclear. 
We prove that for such $X$ the free LCS $L(X)$ has the stronger property of not being multi-Hilbert. 
We deduce that if $X$ is a $k$-space, then the following properties are equivalent: (1) $L(X)$ is strongly nuclear; (2) $L(X)$ is nuclear; (3) $L(X)$ is multi-Hilbert; (4) $X$ is countable and discrete. 
On the other hand, we show that $L(X)$ is strongly nuclear for every projectively countable $P$-space (in particular, for every Lindel\"of $P$-space) $X$.

We observe that every Schwartz LCS is multi-reflexive. It is known that 
if $X$ is a $k_\omega$-space, then $L(X)$ is a Schwartz LCS \cite{Chasco}, hence  $L(X)$ is multi-reflexive. We show that for any first-countable paracompact
 (in particular, metrizable) space $X$ the converse is true, so $L(X)$ is multi-reflexive if and only if $X$ is a $k_\omega$-space,
equivalently, if $X$ is a locally compact and $\sigma$-compact space.

Similarly, we show that for any first-countable paracompact space $X$
 the free abelian topological group $A(X)$ is a Schwartz group 
%(in the sense of \cite{Chasco}) 
if and only if $X$ is a locally compact space such that the set $X^{(1)}$ of all non-isolated points of $X$ is $\sigma$-compact. 
\end{abstract}



%%%%%%%%%%%%%%%%%%%%%%%%%%%%%%%%%%%%%%%%%%

\maketitle
\section{Introduction}\label{s:intro}
\bigskip
%This version is an extended and substantially improved revision of the previous version of our paper.
%\footnote{The main differences between two versions are the following. We realized that $c$-nuclear LCS  which we introduced earlier are exactly the Schwartz LCS.
%Therefore, several publications devoted to the study of Schwartz spaces are relevant to our current research and they have been used in the updated version of our paper.
%The list of references is enlarged accordingly.  }

All topological spaces are assumed to be Tychonoff and all topological vector spaces and groups are Hausdorff.
% and infinite unless explicitly stated otherwise.
All vector spaces are considered over the field of real numbers $\mathbb R$.
The abbreviation LCS means locally convex space.

Every Tychonoff space $X$ can be considered as a subspace of its {\it free locally convex space}
$L(X)$ characterized by the following property: every continuous mapping $f:X\to E$ to an LCS
$E$ uniquely extends to a continuous linear mapping $\widehat f:L(X)\to E$. 

Similarly, one defines the {\it free abelian group} $A(X)$ over a Tychonoff space $X$: $X$ is a subspace of generators of $A(X)$
 and the topology of $A(X)$ is 
characterized by the property that every continuous mapping $f:X\to G$ to an abelian topological group $G$
uniquely extends to a continuous homomorphism $\widehat f:A(X)\to G$. 
It is known that $A(X)$ naturally embeds into $L(X)$ (\cite{Tk}, \cite{Usp1}).

Free topological groups were introduced in 1941 by A. A. Markov in the short note \cite{Markov1}.
 The complete construction appeared 4 years later, in \cite{Markov2},
where certain basic properties of free topological group $F(X)$ and free abelian topological group $A(X)$ over a Tychonoff space $X$ were established.
In particular, Markov answered in the negative Kolmogorov's question as to whether every topological group is a normal space.
Several proofs of the existence of free topological groups have been given,
it seems that the shortest one is due to S. Kakutani \cite{Kakutani2}.
 Another proof, due to M. I. Graev \cite{Graev1}, \cite{Graev2} is based on extending continuous pseudometrics from
a Tychonoff space $X$ to invariant continuous pseudometrics on $F(X)$ (or $A(X)$).
Free topological vector spaces were mentioned by A. A. Markov \cite{Markov1}, without any details. Later, in 1964, D. A. Raikov constructed the free locally convex space $L(X)$ for a uniform space $X$ \cite{Raikov}.

Free LCS constitute a very important subclass of locally convex spaces.
It suffices to mention that every locally convex space is a linear quotient of a free LCS
and that every Tychonoff space $X$ embeds as a \emph{closed} subspace into $L(X)$.
Analogous basic facts are valid for $F(X)$ and $A(X)$.

Given a class $\sP$ of Banach spaces, an LCS $E$
is called {\em multi-$\sP$} if $E$ can be isomorphically embedded in a product of spaces in $\sP$.
The following result was obtained in \cite{Usp08}: for every Tychonoff space $X$ the free LCS
 $L(X)$ can be isomorphically embedded in the product of Banach spaces of the form $l^1(\Gamma)$, in other words,
 $L(X)$ is multi-$\sL^1$, where $\sL^1$ is the class of all Banach spaces of the form
$l^1(\Gamma)$. It follows that, 
% and therefore, 
as a topological group, $L(X)$ admits a topologically faithful unitary representation, 
that is, $L(X)$ is isomorphic to a subgroup of the unitary group. 

Vladimir Pestov raised the following
\begin{question}\label{q:pestov}
When is $L(X)$ {multi-Hilbert}?
\end{question}

The question was motivated by the fact that $L(X)$ is nuclear, hence multi-Hilbert,
if $X$ is a countable
discrete space \cite[Ch. 3, Theorem 7.4]{Sch}; 
in that case $L(X)$ is the locally convex direct sum
of countably many one-dimensional spaces, and %usually is denoted in the literature by $\varphi$. 
will be denoted by $\varphi$.

In Section \ref{s:nuclear} we prove that $L(X)$ is not multi-Hilbert 
%and hence $L(X)$ is non-nuclear 
whenever $X$ contains an infinite compact subset.
As a consequence we deduce that if $X$ is a $k$-space, then the following properties are equivalent:
 (1) $L(X)$ is strongly nuclear; (2) $L(X)$ is nuclear; (3) $L(X)$ is multi-Hilbert; (4) $X$ is countable and discrete. 
On the other hand, we show that $L(X)$ is strongly nuclear (hence nuclear) for every projectively countable $P$-space (in particular, for every Lindel\"of $P$-space) $X$.

Recall the definition of nuclear maps and nuclear spaces.
\begin{definition}\label{d:nuclear}
% \cite[Ch. 3, \S 7]{Sch}.
 A linear map $u:E\to B$
from an LCS $E$ to a Banach space $B$ is called {\em nuclear} if it can be written in the form
$$
u(x)=\sum_{n=1}^\infty \lambda_n f_n(x) y_n, 
$$
where $\sum |\lambda_n|<+\infty$, $(f_n)$ is an equicontinuous sequence of linear functionals on $E$, 
and $(y_n)$ is a bounded sequence in $B$. An LCS $E$ is called {\em nuclear} if any continuous linear map from
$E$ to a Banach space is nuclear.
\end{definition}

The definition of nuclear LCS and the following important permanence results about nuclear LCS are due to  A. Grothendieck \cite{Grot2} (see also \cite[Ch. 3, \S 7]{Sch}).
\begin{itemize}
\item Every subspace of a nuclear LCS is nuclear.
\item Every Hausdorff quotient space of a nuclear LCS is nuclear.
\item The product of an arbitrary family of nuclear LCS is nuclear.
\item The locally convex direct sum of a countable family of nuclear LCS is a nuclear space. 
\item Every nuclear LCS is multi-Hilbert.
\end{itemize}
\noindent
Note also
\begin{itemize}
\item The class of multi-Hilbert LCS is closed under Hausdorff quotients.
\end{itemize}

Lemma \ref{l:factor_hilbert} below also should be attributed to A. Grothendieck.
For the definition of a {\em nuclear group} and a study of free abelian nuclear groups we refer the reader to \cite{Aus}.
Note here only that the class of nuclear groups contains all nuclear LCS.
 
Michael Megrelishvili asked whether $L(X)$ is {\em multi-reflexive}.
It turns out that this problem is tightly connected with the properties of the Schwartz locally convex spaces.
Schwartz LCS play an important role in analysis and its applications.
Remark that in several aspects the properties of Schwartz LCS resemble the finite-dimensional Banach spaces,
for instance, their bounded subsets are precompact. The latter property for Banach spaces is equivalent to finite-dimensionality.
The notion of Schwartz LCS was introduced by A. Grothendieck in \cite{Grot1}.
In this paper we will use the following equivalent versions: 

\begin{definition}\label{d:Sch}
An LCS $E$ is called a {\em Schwartz space} if one of the following equivalent conditions holds:
\begin{itemize}
\item [(a)] Every continuous linear map from $E$ to a Banach space
is compact, that is, sends a neighborhood of zero to a set with a compact closure;
\item [(b)] For every neighborhood $U$ of zero in $E$, there exits another neighborhood $V$ of zero 
such that for every $\lambda > 0$  the set $V$ can be covered by finitely many translates of $\lambda U$.
\end{itemize}
\end{definition}

For other equivalent definitions of a Schwartz space see \cite[Section 3.15]{Horvath}. 
Several relevant problems in the theory of Schwartz LCS were solved in 
\cite{Bel1}, \cite{Bel2}, \cite{Randtke1}, \cite{Randtke2}, \cite{Terzioglu} (our list of references does not attempt to be complete).
Note that the class of all Schwartz LCS (as well as the classes of all strongly nuclear, all nuclear, all multi-Hilbert and all multi-reflexive LCS) constitutes a {\em variety}, i.e.
is closed under the formation of subspaces (not necessarily closed), Hausdorff quotients, arbitrary products and isomorphic images (see \cite{DMS}).
%A group version of the concept of a Schwartz space is introduced and studied in the paper \cite{Chasco} (see also \cite{Aus2}).

We observe that every Schwartz LCS is multi-reflexive. This follows, for example, from 
\cite[17.2.9 and 21.1.1(c)]{Jarch}; we give a short proof of this assertion in Theorem~\ref{t:cnuc}.
It is known that $L(X)$ is a Schwartz space, hence multi-relexive, for every $k_\omega$-space $X$
\cite[Theorem 5.2]{Chasco}. In Section \ref{s:multi} we provide an alternative proof of this fact (Theorem~\ref{t:lx}). 

Recall that a topological space $X$ satisfies the {\it first axiom of countability} if
each point of $X$ has a countable base of neighborhoods.
In Section \ref{s:multi} we give a complete answer to the question: when is $L(X)$ multi-reflexive?
 --- for paracompact spaces $X$ satisfying the first axiom of countability; namely, 
if $X$ is such a space, then the following properties are equivalent:
(1) $L(X)$ is a Schwartz space; (2) $L(X)$ is multi-reflexive; (3) $X$ is locally compact and $\sigma$-compact
 (= $X$ is the union of countably many compact sets).
 In particular, for a metrizable space $X$,
$L(X)$ is multi-reflexive if and only if $X$ is locally compact space and has a countable base. 
 We also give a short direct proof of the fact that $L(X)$ is not multi-reflexive if
$X$ is the space of irrational numbers (Example~\ref{e:P}).

A group version of the concept of a Schwartz space is introduced and studied in the paper \cite{Chasco} (see also \cite{Aus2}).
In Section \ref{s:A(X)} we give a complete answer to the question: when is the free abelian group 
$A(X)$ a {\it Schwartz group}? ---
 for paracompact spaces $X$ satisfying the first axiom of countability. 
In that case,
$A(X)$ is a Schwartz group 
if and only if $X$ is a locally compact space such that the set $X^{(1)}$ of all non-isolated points of $X$ is $\sigma$-compact. 

Our notations are standard, the reader is advised to consult with the monographs  \cite{Eng}, \cite{Fabian}, \cite{Sch}
 for the notions which are not explicitly defined in the text. 
In the end of the article we pose several open questions.
 
%%%%%%%%%%%%%%%%%%%%%%%%%%
\section{Nuclear and multi-Hilbert $L(X)$} \label{s:nuclear}
\bigskip
Let $X$ be a Tychonoff space containing an infinite compact subset $K$.
Every continuous pseudometric defined on the compact space $K$ extends to a continuous pseudometric on $X$, therefore $L(K)$
can be identified by a linear topological isomorphism with a  subspace of $L(X)$ \cite{Usp1}. Since the free abelian group $A(K)$ naturally embeds into $L(K)$,
 we observe that $A(K)$ is isomorphic to a topological subgroup of $L(X)$. Hence, if $L(X)$ is a nuclear LCS,
then $A(K)$ is a nuclear group. However, the free abelian group $A(K)$ is nuclear if and only if a compact space $K$ is finite \cite[Theorem 7.3]{Aus}.
This observation leads to

\begin{proposition}
\label{p: observ}
If a Tychonoff space $X$ contains an infinite compact subset, then $L(X)$ is not nuclear.
\end{proposition}

Below we prove a significantly stronger result, the proof of which requires %much 
more elaborate arguments from functional analysis.

\begin{theorem}
\label{t:main}
If a Tychonoff space $X$ contains an infinite compact subset, then $L(X)$ is not multi-Hilbert.
\end{theorem}

In order to prove Theorem \ref{t:main} we need some preparations. %to do a preparatory work.
Let us say that a subset $C$ of a Banach space $E$ is an {\em ellipsoid} if $C$ is the image
of a closed ball in a Hilbert space $H$ under a bounded linear mapping $H\to E$.
Note that ellipsoids are weakly compact, hence closed in $E$.

\begin{lemma}
\label{l:1}
There exists a Banach space $E$ and a countable sequence $S=(x_n)$ converging to zero in $E$ such that 
$S$ is not contained in any ellipsoid in $E$. 
\end{lemma}
 
\begin{proof}
Let $E$ be a Banach space without the approximation property.
According to \cite[Theorem 1.2]{Fonf},
there is a compact subset $K$ in $E$ such that for any Banach space $F$ with a basis and any 
injective bounded operator $T:F\to E$
the image $T(F)$ does not contain $K$. Then $K$ is not contained in any ellipsoid. Indeed,
if $H$ is a Hilbert space and $A:H\to E$ is a bounded linear operator, we can write $A$ as the composition $H\to \breve{H}\to E$, where
$\breve{H}$ is the quotient of $H$ by the kernel of the operator $A$. Then $\breve{H}$ is a Hilbert space and the operator $T:\breve{H}\to E$
is injective, hence $K$ is not contained in the image $T(\breve{H})=A(H)$.

Now we find a countable sequence $S=(x_n)$ in $E$ converging to zero and such that $K$ is contained in the closed 
convex hull of $S$. If an ellipsoid contains $S$, then it also contains $K$, which is impossible.
\end{proof}

Note that in fact a sequence with properties stated in Lemma~\ref{l:1} can be found in any Banach space
$E$ not isomorphic to a Hilbert space. For the proof one can use, for example,
the following: in every Banach space which is non-isomorphic to a Hilbert space there is a sequence $F_n$ of finite-dimensional subspaces 
whose Banach-Mazur distances from the Euclidean spaces of the same dimension go to infinity \cite{Joichi}.

\begin{lemma}
\label{l:factor_hilbert}
Let $L$ be a multi-Hilbert LCS. Then every continuous linear map $f: L\to E$ from $L$ to a Banach space $E$
can be represented as a composition $f=p\circ g$, where $g:L\to H$ and $p:H\to E$ are two continuous linear maps
and $H$ is a Hilbert space.
\end{lemma}
\begin{proof} We assume that $L$ is isomorphically embedded into the product $\prod_{i\in I} E_i$, where all $E_i$ are Hilbert spaces.
Denote by $V$ the unit open ball of $E$. Since $f$ is continuous, and by the definition of the product topology,
there is a finite face $E_A =\prod_{i\in A} E_i$ of the product $\prod_{i\in I} E_i$, and a neighborhood $U$ of zero in $E_A$,
 such that $f(p^{-1}(U))$ is contained in $V$.
Here the set of indexes $A\subset I$ is finite and $p: X \to E_A$ is the corresponding projection.

We claim that  $\ker(p) \subset \ker(f)$. Indeed, if $x \in L$ and $p(x)=0$, then $p(tx)= 0$ for every scalar $t$, so $tx \in p^{-1}(U))$.
This means that $f(tx) =t(f(x))$ belongs to the same unit ball $V$ for every scalar $t$, which implies that $f(x)=0$.
 So there is some map $g: p(X)\to E$ such that $f=g\circ p$. Clearly, $g$ is a linear map. It is also bounded, because $g(U)= f(p^{-1}(U))$
 was assumed to be contained in $V$. The map $g$ is defined on the vector subspace $p(X)$ of the Hilbert space $E_A$.
 (The LCS $E_A$ is isomorphic to a Hilbert space being a finite product of Hilbert spaces). 
Further, the map $g$ can be extended by continuity to the closure of that vector space in $E_A$.
A closed vector subspace of a Hilbert space is  Hilbert itself, and we get the required factorization.
\end{proof}

\begin{fact} 
\label{l:sequence}
Let $S$ be a non-trivial countable convergent sequence. 
 Then $L(S)$ is not multi-Hilbert.
\end{fact}

\begin{proof}
On the contrary, suppose that $L(S)$ is multi-Hilbert. By Lemma~\ref{l:1} we define
a 1-to-1 continuous map $f:S\to E$ such that $f(S)$ is not contained in an ellipsoid.
Extend $f$ to a linear continuous map $\widehat f: L(S)\to E$.
Now, by Lemma~\ref{l:factor_hilbert}, we can represent $\widehat f$
as a composition $\widehat f=p\circ \widehat g$, where $\widehat g:L(S) \to H$ and $p:H\to E$ are two continuous linear maps,
$\widehat g$ extends the map $g:S\to H$ and $H$ is a Hilbert space. The following diagram illustrates our construction. 

$$
\xymatrix{
&H \ar[d]_p\\
S\ar[r]_f \ar[ur]^g&E%\ar@/_/@{-->}[u]_s
}
\qquad \qquad\qquad
\xymatrix{
&H \ar[d]_p\\
L(S)\ar[r]_{\widehat f} \ar[ur]^{\widehat g}&E
}
$$
The compact set $g(S)\sbs H$ is contained in a ball, hence $f(S)=p(g(S))$ is contained in
an ellipsoid. The obtained contradiction finishes the proof.
\end{proof}
\noindent
{\bf Proof of Theorem \ref{t:main}.}
Let $K$ be an infinite compact subset of $X$ and assume that $L(X)$ is multi-Hilbert.
As we have noted before $L(K)$ can be identified with a linear subspace of $L(X)$, hence $L(K)$ is multi-Hilbert. 
$K$ is an infinite compact space, therefore we can choose a continuous map $\pi:K \to [0,1]$ with an infinite image.
Let $M =\pi(K)$. 
We have that $\pi:K \to M$ is a closed (hence, quotient) continuous surjective mapping. 
In this situation $\pi$ lifts to a linear continuous quotient mapping $\widehat \pi$ from $L(K)$ onto $L(M)$.
The latter fact can be proved analogously to \cite[Proposition 1.8]{Okunev}.
Since the class of multi-Hilbert LCS is closed under Hausdorff quotients we conclude that $L(M)$ is also multi-Hilbert.
However, $M$ is an infinite compact subset of the segment $[0,1]$, therefore $M$ contains a copy of a convergent sequence $S$.
This would mean that $L(S)$ is multi-Hilbert contradicting Fact \ref{l:sequence}. The obtained contradiction finishes the proof of Theorem \ref{t:main}.
\,\,\,\,\,$\Box$

Recall that a topological space $X$ is called a {\it $k$-space} whenever $F \subset X$ is closed iff the intersection $F \cap K$ is closed in $K$
for every compact $K \subset X$. If $X$ can be covered by countably many compact subsets $K_n$ 
such that $F\sbs X$ is closed iff all intersections $F\cap K_n$ are closed, then $X$ is said to be
a {\it $k_\omega$-space} \cite{k_omega}.

An LCS $E$ is called {\em strongly nuclear} \cite[1.4]{Randtke1} if for every continuous seminorm $q$ on $E$
there exist a rapidly decreasing sequence $(\lambda_n)$ (this means that $\sum n^k|\lambda_n|<\infty$
for every $k$) and an equicontinuous sequence $(a_n)$ of linear functionals on $E$ such that
$q(x)\le \sum |\lambda_n|\langle x,a_n\rangle|$ for all $x\in E$. It is easy to verify that 
$L(\mathbb N)=\varphi$ is strongly nuclear, where 
$\mathbb N$ denotes the discrete space of natural numbers.

Following \cite{DMS} denote by $\sV(\varphi)$ the variety of LCS generated by $\varphi$.
Note that $\sV(\varphi)$ is the (unique) second smallest variety, in the sense that every
variety properly containing $\sV(\R)$ contains $\sV(\varphi)$ \cite{DMS}.

\begin{corollary} \label{cor:1} Let $X$ be a $k$-space. The following conditions are equivalent:
\begin{itemize}
\item [(i)] $L(X)$ is strongly nuclear;
\item [(ii)] $L(X)$ is nuclear;
\item [(iii)] $L(X)$ is multi-Hilbert;
\item [(iv)] $X$ is a countable discrete space. 
\end{itemize}
\end{corollary}
\begin{proof} Only the implication (iii) $\implies$ (iv) is new and needs to be proved.
By Theorem \ref{t:main}, every compact subset of $X$ is finite, therefore the $k$-space $X$ is discrete.
Every multi-Hilbert space is multi-reflexive, and we will show later
that $L(X)$ is not multi-reflexive if $X$ is an uncountable discrete space (Fact \ref{l:exampleD}).
\end{proof}

Besides the requirement that $X$ should contain no
infinite compact subspaces, there is another necessary condition for $L(X)$ to be nuclear.

\begin{theorem} \label{t:2}
Let $L(X)$ be a nuclear LCS. Then every metrizable image of $X$ under a continuous map must be separable.
\end{theorem}
\begin{proof} Let $f:X\to M$ be a continuous map onto a metrizable space $M$.
Assume that the topology of $M$ is generated by a metric $d$. Denote by $B=C_{b}(M)$ the Banach space
 of all bounded continuous real-valued functions on $M$ equipped with the supremum norm.
It is known that the Kuratowski mapping isometrically embeds $M$ into the Banach space $B$.
Specifically, if $(M,d)$ is a metric space, $a$ is a fixed point in $M$, then the mapping $\Phi :M\rightarrow C_{b}(M)$
defined by
$$\Phi (x)(y)=d(x,y)-d(a,y)\quad {\mbox{for all}}\quad x,y\in M,$$
is an isometry.

Extend the map $f:X\to M$ to a linear continuous
map $\widehat f:L(X)\to B$. It follows from the definition of a nuclear map that 
the range of $\widehat f$ is separable. Since $M$ is contained in that range, it is separable, too.
\end{proof}

It follows from Theorem~\ref{t:3} below that there exist non-separable spaces $X$ such that 
$L(X)$ are strongly nuclear, hence nuclear. We say that a Tychonoff space $X$ is {\em projectively countable} if every metrizable image 
of $X$ under a continuous map is countable.  Recall that 
a {\em $P$-space} is a topological space in which the intersection of countably many open sets is open.
To see that Lindel\"of $P$-spaces are projectively countable, observe that if $X$ is a $P$-space, all points of $Y$ are $G_{\delta}$,
and $f: X \to  Y$ is a continuous map, then the inverse
images of points from $Y$ form a disjoint open cover of $X$.
It is known that {\it $\omega$-modification} of a topology of any scattered Lindel\"of space produces a Lindel\"of $P$-space.
A concrete example of an uncountable Lindel\"of $P$-space is the so-called {\it one-point Lindel\"ofication} $X=\Gamma \cup \{p\}$ of an uncountable discrete 
space $\Gamma$. All open neighborhoods of the point $p$ in X are of the form $A \cup \{p\}$, where $A \subset \Gamma$ and $|\Gamma \setminus A| \leq \aleph_0$.

\begin{theorem}
\label{t:3}
If $X$ is a projectively countable $P$-space (in particular, a Lindel\"of $P$-space), then $L(X)$ is strongly nuclear.
\end{theorem}
\begin{proof}
Let $M$ be a metrizable space and $f:X\to M$ be a continuous map. Then $f(X)$ is countable.
Every point $m \in M$ is a $G_{\delta}$-set in $M$,
therefore every fiber $f^{-1}(m)$ is open in the $P$-space $X$. 
We see that $X$ is the disjoint countable union of open fibers $f^{-1}(m)$, $m\in f(X)$.
It follows that every continuous map $f:X\to M$ admits a factorization $X\to \mathbb N\to M$, where $\mathbb N$ is the
discrete space of natural numbers, hence any linear continuous map $L(X)\to B$ to a Banach space $B$ 
admits a factorization $L(X)\to L(\mathbb N)\to B$.
%Since $L(\mathbb N)=\varphi$ is strongly nuclear,
% so is the map $L(X)\to B$. 
Using \cite[Remark 3.4]{Randtke1} we can conclude that $L(X)$ can be isomorphically embedded into a power of the space $\varphi$.
This means that $L(X)$ belongs to $\sV(\varphi)$ which is a sub-variety of the variety of all strongly nuclear spaces \cite{DMS},
since $L(\mathbb N)=\varphi$ is strongly nuclear.
\end{proof}

It turns out that there are projectively countable $P$-spaces which are not Lindel\"of $P$-spaces. 
A Tychonoff space $X$ is called {\em cellular-Lindel\"of} if for every disjoint family $\sU$
of open nonempty subsets of $X$ there is a Lindel\"of subspace $L$ of $X$ such that $L$ 
meets every member of the family $\sU$. It is easy to see that every cellular-Lindel\"of $P$-space
is projectively countable. Hence, Theorem \ref{t:3} implies

\begin{proposition}
\label{p:2}
If $X$ is a cellular-Lindel\"of $P$-space, then $L(X)$ is a strongly nuclear LCS.
\end{proposition}

Another quick consequence provides sufficient conditions %implying the nuclear property
for the free abelian group $A(X)$ to be nuclear:

\begin{corollary}
\label{cor:2}
If $X$ is a cellular-Lindel\"of $P$-space, then $A(X)$ is a nuclear topological group.
\end{corollary}

See \cite[Theorem 3.16]{Tkachuk} for an example of 
a cellular-Lindel\"of $P$-space which is not even weakly Lindel\"of.
%%%%%%%%%%%%%%%%%%

\section{Schwartz and Multi-reflexive $L(X)$}\label{s:multi}
\bigskip
Recall that an LCS is {\em multi-reflexive} if it is isomorphic to a subspace of a product
of reflexive Banach spaces. For instance, every LCS endowed with its weak topology is multi-Hilbert (hence multi-reflexive),
 since it embeds into a product of the real lines. 
However, a reflexive LCS need not be multi-reflexive \cite[Remark 4.14 (a)]{Megr}.

Our proof of Theorem~\ref{t:main} is based on the fact that certain compact subsets of Banach spaces
are not contained in ellipsoids. However, 
every weakly compact subset of a Banach space is contained
in the image of the closed unit ball of a reflexive space under a bounded linear mapping (see \cite[Theorem 13.22]{Fabian}).
 This motivates the following question (suggested to us by Michael Megrelishvili):

\begin{question}\label{q:megrel}
When is $L(X)$ multi-reflexive?
\end{question}

We have seen that $L(X)$ is not multi-Hilbert if $X$ is an infinite compact space 
(Theorem~\ref{t:main}). 
Nevertheless, $L(X)$ is a Schwartz LCS for every compact space $X$~\cite{Chasco}. We provide an alternative proof of this statement
(Theorem~\ref{t:lx}) and due to the fact that every Schwartz LCS is multi-reflexive (Theorem~\ref{t:cnuc}), we
conclude that $L(X)$ is multi-reflexive for every compact space $X$ (Theorem~\ref{t:main2}).

In general, $L(X)$ need not be Schwartz or multi-reflexive.
 We provide  a complete characterization of first-countable paracompact spaces $X$ such that
$L(X)$ is Schwartz or multi-reflexive (Theorem \ref{t:main3}).
In particular, if $X$ is metrizable,
then $L(X)$ is multi-reflexive iff $L(X)$ is a Schwartz LCS iff $X$ is locally compact 
and has a countable base (Corollary~\ref{c:main}).

\begin{lemma}
\label{l:fac}
If $L$ is a Schwartz LCS, then every continuous linear map $L\to B$ to a Banach space $B$
admits a factorization $L\to B_1\to B$, where $B_1$ is a Banach space,
 $L\to B_1$ is a continuous linear map  and $B_1\to B$ is a compact operator.
\end{lemma}

\begin{proof}
Let $U$ be a balanced convex neighborhood of zero in $L$ such that the image of $U$ in $B$ has
a compact closure. Let $p$ be the gauge functional corresponding to $U$. We can take for $B_1$
the completion $\widetilde L_U$ of the normed space $L_U$ associated with the seminorm $p$, as in 
\cite[Ch. 3, \S 7]{Sch}.
\end{proof}

\begin{theorem}
\label{t:cnuc}
Every Schwartz LCS is multi-reflexive.
\end{theorem}

\begin{proof}
Let $L$ be a Schwartz LCS. We want to prove that continuous linear maps from $L$ to reflexive
Banach spaces generate the topology of $L$. Since maps of $L$ to all Banach spaces generate the
topology of $L$, 
it suffices to prove that every continuous linear map $L\to B$
to a Banach space admits a factorization $L\to B_1\to B$, where $B_1$ is a reflexive Banach space.
According to Lemma~\ref{l:fac}, we can find a factorization $L\to B_2\to B$, 
where $B_2\to B$ is a compact operator between Banach spaces.
 Compact operators are weakly compact, and in virtue of the 
Davis--Figiel--Johnson--Pe\l{}czy\'nski theorem \cite[Theorem 13.33]{Fabian} weakly compact operators
between Banach spaces admit a factorization through reflexive spaces. Thus we have a required factorization
$L\to B_2\to B_1\to B$ for some reflexive Banach space $B_1$.
\end{proof}

Below we collect several basic facts about the topology of $L(X)$ and associated Banach spaces.
For a space $X$ we denote by $L_0(X)$ the hyperplane
$$
L_0(X)=\left\{\sum c_i x_i: \sum c_i = 0, \ c_i\in \R, \ x_i\in X\right\}
$$ 
of $L(X)$.
The topology of $L_0(X)$ is generated by the seminorms $\overline d$ of the following form.
For a continuous pseudometric $d$ on $X$, the seminorm $\overline d$ on $L_0(X)$ is defined by 

\begin{equation}
\label{eq:L1}
\overline d(v)=\inf\left\{\sum |c_i| d(x_i,y_i): v=\sum c_i(x_i-y_i), \ c_i\in \mathbb R, \ x_i, y_i\in X\right\}.
\end{equation}

Denote by $B_d$ the Banach space associated with the seminormed space $(L_0(X), \overline d)$, that is, 
the completion of the quotient by the kernel of $\overline d$. The dual Banach space $B_d^*$ is naturally
isomorphic to the quotient $C_d/\mathbb R$ of the space of $d$-Lipschitz functions on $X$ 
by the one-dimensional subspace of constant functions. The seminorm on $C_d$ assigns to each
$d$-Lipschitz function $f$ its Lipschitz constant 

\begin{equation}
\label{eq:L2}
\norm{f}= \inf\left\{k\ge 0: |f(x)-f(y)|\le k\cdot d(x,y)\,\mbox{for all}\, x, y\in X\right\}.
\end{equation}

\begin{lemma}
\label{l:sqrt}
If $X$ is a compact space and $d$ is a continuous pseudometric on $X$, then the natural operator
$B_{\sqrt{d}}\to B_d$ is compact.
\end{lemma}

\begin{proof}
We may assume that $d$ is a metric. Pick a point $x_0\in X$, and identify the quotient
$C_d/\mathbb R$ with the Banach space $\breve{C}_d=\{f\in C_d: f(x_0)=0\}$. A bounded linear operator between Banach spaces
is compact if and only if its dual is compact, by the Shauder theorem \cite[Theorem 15.3]{Fabian}, so it suffices to prove
that the embedding $\breve{C}_d\to \breve{C}_{\sqrt{d}}$ is compact. Take any sequence $(f_n)$
in the unit ball of $\breve{C}_d$, that is, a sequence of $d$-non-expanding functions such that $f_n(x_0)=0$.
By the Arzel\`{a}--Ascoli theorem, the sequence $(f_n)$ has a uniformly convergent susbsequence. To simplify
the notation, assume that the sequence $(f_n)$ itself uniformly converges to a limit $f$.
We claim that $(f_n)$ converges to $f$ in $\breve{C}_{\sqrt{d}}$ as well. Let $\e > 0$ be given. 
Put $g_n=f_n-f$. We must prove that for $n$ large enough we have
\begin{equation}
\label{eq1}
|g_n(x)-g_n(y)|\le \e \sqrt{d(x,y)}
\end{equation}
for all $x,y\in X$. Note that each $g_n$ is 2-Lipschitz with respect to $d$. If $x$ and $y$
are close to each other, namely, if $d(x,y)\le \e^2/4$, then (\ref{eq1}) holds:
$$
|g_n(x)-g_n(y)|\le 2d(x,y)=2\sqrt{d(x,y)}\sqrt{d(x,y)}\le\e\sqrt{d(x,y)}.
$$
On the other hand, for pairs $x,y$ such that $d(x,y)>\e^2/4$ the inequality~\ref{eq1}
holds for $n$ large enough because the sequence $(g_n)$ uniformly converges to zero. 
\end{proof}

\begin{theorem}
\label{t:lx}
If $X$ is a compact space, then $L(X)$ is a Schwartz LCS. 
\end{theorem}
 
\begin{proof}
Let $T:L(X)\to B$ be a continuous linear operator, where $B$ is a Banach space. We want to prove that $T$ is compact.
Let $d$ be the pseudometric on $X$ induced by $T$ from the metric on $B$. We have a factorization
$L(X)\to B_d\to B$ of the operator $T$. By Lemma~\ref{l:sqrt}, the middle arrow in the factorization 
$L(X)\to B_{\sqrt{d}}\to B_d\to B$ is a compact operator. It follows that $T$ is a compact operator as well. 
\end{proof}

Combining Theorems \ref{t:cnuc} and \ref{t:lx}, we get the result we were aiming at.

\begin{theorem}
\label{t:main2}
If $X$ is a compact space, then $L(X)$ is multi-reflexive. 
\end{theorem}

Theorems~\ref{t:lx} and~\ref{t:main2} can be readily generalized to the case of 
$k_\omega$-spaces, similarly to~\cite{Chasco}. Indeed, let $X=\bigcup_{n\in\N} K_n$ be a $k_\omega$-space, 
where $(K_n)$ is a sequence of compact subsets of $X$ witnessing the $k_\omega$-property.
Then $X$ is a quotient of the free sum $\bigoplus_{n\in\N} K_n$ \cite{k_omega}, and $L(X)$ is a quotient of 
$L(\bigoplus_{n\in\N} K_n)$ which is isomorphic to the countable locally convex sum of $L(K_n)$.
Since Schwartz spaces are preserved by quotients and by locally convex countable sums \cite[Proposition 21.1.7]{Jarch},
 we conclude that $L(X)$ is a Schwartz LCS.
%locally compact
%$\sigma$-compact spaces ($\sigma$-compact = the union of countably many compact sets).

In particular, we have established the following (which is a somewhat special case of~\cite{Chasco}):
\begin{theorem}
\label{t:loccomp}
If $X$ is a locally compact $\sigma$-compact space, then $L(X)$ is a Schwartz LCS and hence
(Theorem~\ref{t:cnuc}) multi-reflexive. 
\end{theorem}
%
%\begin{proof}
%If suffices to consider the case when $X$ is a disjoint sum $\bigoplus K_n$ of compact spaces.
%Indeed, in the general case $X$ is a quotient of such a disjoint sum, so we get a quotient map
%$L(\bigoplus K_n)\to L(X)$, and it is clear that Schwartz spaces are preserved by quotients.
%So assume $X=\bigoplus K_n$. Then $L(X)$ is the locally convex sum of the spaces $L(K_n)$
%which are Schwartz by Theorem~\ref{t:lx}. Therefore Lemma~\ref{l:freesum} below 
%finishes the proof.
%\end{proof}
%
%\begin{lemma}
%\label{l:freesum}
%The locally convex sum of countably many Schwartz spaces is Schwartz.
%\end{lemma}
%
%\begin{proof}
%Let $(V_n)$ be a sequence of Schwartz spaces, $V=\bigoplus V_n$ their locally convex sum.
%We want to prove that any linear operator $f:V\to B$ to a Banach space is compact. Basic neighborhoods
%of zero in $V$ are of the following form: take a neighborhood $U_n$ of zero in $V_n$, and consider
%the convex hull of the sum $\bigoplus U_n$. Since each $V_n$ is Schwartz, 
%we can pick the neighborhoods $U_n$ so that $f(U_n)$ has a compact closure in $B$. We may additionally
%assume that the diameter of each $f(U_n)$ is less that $\e_n$, where $\sum \e_n<\infty$.
%Every point in the convex hull $W$ of the union $\bigcup f(U_n)$ is at the distance less that 
%$\sum_{k=n+1}^\infty \e_k$ from convex hull of the union $\bigcup_{k\le n} f(U_k)$ 
%which is precompact. 
%It follows that $W$ is precompact and that $f$ is a compact operator.
%\end{proof}

Our next aim is to show that for a wide and natural class of Tychonoff spaces $X$ the converse of Theorem~\ref{t:loccomp} is true:
$L(X)$ is a Schwartz LCS if and only if $X$ is a locally compact and $\sigma$-compact space.

\begin{lemma}
\label{l:2}
If $L(X)$ is multi-reflexive, then for every Banach space $E$ and every continuous map $f:X\to E$
the image $f(X)$ can be covered by countably many weakly compact subspaces of $E$.
\end{lemma}

\begin{proof}
If $L(X)$ is multi-reflexive, by literally the same arguments as we have used in the proof of
Lemma \ref{l:factor_hilbert},
we can represent $f$ as a composition $f=p\circ g$, where $g:X\to F$
is a continuous map to a reflexive Banach space $F$ and $p:F\to E$ is a bounded linear map. Since the reflexive Banach space $F$ is
the union of countably many weakly compact sets, the same is true for the set $p(F)$ which is a subset
of $E$ containing $f(X)$.
\end{proof}

\begin{fact}
\label{l:exampleD}
If $X$ is an uncountable discrete space, then $L(X)$ is not multi-reflexive.
\end{fact}

\begin{proof}
%This time we take 
Consider the Banach space $E=l^1(\Gamma)$, 
where $\Gamma$ is a set of indexes with the same cardinality as $X$.
%Consider the set $Y=\left\{e_{\gamma}: \gamma\in\Gamma\right\}$, where
% $e_{\gamma}(\alpha) = 1$ if $\alpha = \gamma$, otherwise $e_{\gamma}(\alpha) = 0$. 
%Since $\norm{e_{\gamma} -  e_{\gamma^{\prime}}}_E = 2$ for every $\gamma \neq \gamma^{\prime}$, the set $Y$ is discrete in $E$.
Let $Y=\left\{e_{\gamma}: \gamma\in\Gamma\right\}$ be the canonical basis of $E$.
Let $f$ be any bijection from $X$ onto $Y$, then $f:X\to E$ is a continuous map, because $Y$ is discrete in $E$.
We claim that $Y=f(X)$ cannot be covered by countably many weakly compact subspaces of $E$.
This follows from the easily verified fact that $Y$ is a closed discrete subset of $E$ endowed with the weak topology (use the natural pairing $\langle(a_\g), (b_\g)\rangle=\sum a_\g b_\g$ between
$E=l^1(\Gamma)$ and $E'=l^\infty(\Gamma)$, and consider 0--1 sequences in $E'$).
Alternatively, note that $l^1(\Gamma)$ has the Schur property, that is,
 every weakly converging sequence in $l^1(\Gamma)$ converges in the norm topology, and then
every weakly compact subset of $l^1(\Gamma)$ is compact in the norm topology 
(see \cite[Exercise 5.47]{Fabian}).
Hence, no infinite subset of the closed discrete set $Y$ 
can be covered by a weakly compact subset of $E$, and the claim follows.
\end{proof}

The following known fact plays a crucial role in the sequel.

\begin{theorem}[\cite{Usp1}]
\label{t:Usp1}
Let $X$ be a paracompact (in particular, metrizable) space.
Then for every closed subspace $Y \subset X$, the natural map $L(Y)\to L(X)$ is a topological embedding.
\end{theorem}

\begin{corollary}
\label{l:metr-non-sep}
If $X$ is a non-separable metric space, then $L(X)$ is not multi-reflexive.
\end{corollary}

\begin{proof}
$X$ contains an uncountable closed discrete subspace $Y$, 
so Fact~\ref{l:exampleD} and Theorem \ref{t:Usp1} apply.
\end{proof}

The {\em metric fan} $M$ is the metrizable space defined as follows. As a set, $M$ is the set
$\N\ti\N$ of pairs of natural numbers plus a point $p$ ``at infinity"; $p$ is the only non-isolated
point of $M$, and a basic neighborhood $U_n$ of $p$ consists of $p$ and all pairs $(a, b)$ such that 
$b > n$. Thus, $M$ is the union of countably many sequences converging to $p$. It is easy to verify that the space $M$ 
is not locally compact.
% and a first countable paracompact space is locally compact if and only if it does not contain
%a closed subspace homeomorphic to $M$ (\cite[Lemma 8.3]{vD}).

\begin{fact}
\label{l:fan}
Let $M$ be the metric fan defined above. Then the free LCS $L(M)$ is not multi-reflexive.
\end{fact}

\begin{proof}
Consider the separable non-reflexive Banach space $E=l^1$, and let $(e_n)$ be the canonical basis.
The collection of all vectors of the form $e_n/k$, $k=1,2,\dots$, together with the zero vector,
is a subspace $A\sbs E$ homeomorphic to $M$. Let $f:M\to A$ be a homeomorphism, 
$\widehat f: L(M)\to E$ its linear extension. We want to prove that there is no factorization 
$L(M)\to F\to E$ of $\widehat f$, where $F$ is a reflexive Banach space. Suppose there is such a factorization.
 Let $\breve{A}$ be the image of $M$ in $F$, and let $p$ be the non-isolated point of $\breve{A}$.
Denote the arrow $F\to E$ by $g$. We have continuous bijections $M\to \breve{A}\to A$ such that their composition
is a homeomorphism. It follows that the restriction $g\restriction_{\breve{A}}: \breve{A}\to A$ is a homeomorphism. 
Let $U$ be the closed unit ball in $F$. 
Then $(p+U) \cap \breve{A}$ is a neighborhood of $p$ in $\breve{A}$. 
Since $g\restriction_{\breve{A}}: \breve{A}\to A$ is a homeomorphism, $g((p+U) \cap \breve{A})$ is a neighborhood of $g(p)=0$ in $A$, 
and so is the bigger set $g(U)\cap A$. Thus for $k$ large enough all the vectors $e_n/k$
 lie in the weakly compact set $g(U)$. Multiplying by $k$, we conclude that all the vectors $e_n$
 lie in a weakly compact subset of $E$. That is a contradiction: we noted in the proof of 
Fact~\ref{l:exampleD} that
the set $\{e_n\}$ is closed and discrete with respect to the weak topology of $E$.
\end{proof}

%Combining Theorems~\ref{t:loccomp} and~\ref{t:cnuc} with Lemmas~\ref{l:metr-non-sep}, \ref{l:fan}, 
%and the characterization of locally compact first countable paracompact spaces 
%in terms of the metric fan $M$
%mentioned before Lemma~\ref{l:fan}, we arrive at the main result of the paper:

We now are ready to prove one of the main results of our paper.

\begin{theorem}
\label{t:main3}
Let $X$ be a first-countable paracompact space. The following conditions are equivalent: 
\begin{itemize}
  \item[(i)] $L(X)$ is a multi-reflexive LCS;
  \item[(ii)] $L(X)$ is a Schwartz LCS;
  \item[(iii)] $X$ is locally compact and $\sigma$-compact.
\end{itemize}
\end{theorem}

\begin{proof} (iii) $\implies$ (ii): this is Theorem~\ref{t:loccomp}.

(ii) $\implies$ (i): this is Theorem~\ref{t:cnuc}.

(i) $\implies$ (iii): 
by Theorem \ref{t:Usp1} and Fact~\ref{l:fan}, $X$ does not contain a closed copy of the metric fan~$M$.
A first-countable paracompact space is locally compact if and only if it does not contain
a closed subspace homeomorphic to $M$ (by \cite[Lemma 8.3]{vD}). Hence $X$ is locally compact.
Being locally compact and paracompact, $X$ can be represented as a disjoint union of clopen $\sigma$-compact subspaces
(by \cite[Theorem 5.1.27]{Eng}). Since $X$ does not contain an uncountable closed discrete subspace
(by Fact~\ref{l:exampleD} and Theorem \ref{t:Usp1}), the number of components in the representation of $X$ above is at most countable
and the proof is complete.
\end{proof}

\begin{corollary}
\label{c:main} Let $X$ be a metrizable space. The following conditions are equivalent: 
\begin{itemize}
  \item[(i)] $L(X)$ is a multi-reflexive LCS;
  \item[(ii)] $L(X)$ is a Schwartz LCS;
  \item[(iii)] $X$ is locally compact and has a countable base.
\end{itemize}
\end{corollary}

\begin{example}\label{e:P}
{\em
If $P$ is the space of irrational numbers, then as a straightforward consequence of Theorem~\ref{t:main3} we have that $L(P)$ is not multi-reflexive.
We note that an alternative proof of this fact readily follows from Lemma~\ref{l:2}.

Indeed, consider any separable non-reflexive Banach space $E$. For example, we can take the separable Banach space $E=l^1$.
Then weakly compact subsets of $E$ have empty interior, and it follows from the Baire Category Theorem
that $E$ is not the union of countably many weakly compact sets. On the other hand, $E$, 
like any Polish space, is the image of $P$ under a continuous map. Now apply Lemma~\ref{l:2}.
}
\end{example}

%%%%%%%%%%%%%%%%
\section{Schwartz $A(X)$}\label{s:A(X)}
\bigskip
A concept of a {\em Schwartz topological group} which turned out to be coherent with
the concept of a Schwartz locally convex space 
has been introduced in \cite{Chasco}. Similarly to Schwartz LCS, the class of Schwartz groups enjoys 
analogous permanence properties: it is closed with respect to subgroups, Hausdorff quotients,
products and local isomorphisms. Also, bounded subsets of locally quasi-convex Schwartz groups are precompact \cite{Chasco}.

A natural question arises: when is the free abelian topological group $A(X)$ Schwartz?
 Note that the authors of \cite{Chasco} found sufficient conditions on $X$ implying that $A(X)$ is a Schwartz group,
which are identical to those conditions on $X$ implying that $L(X)$ is a Schwartz LCS.
In this section we give a complete characterization of first-countable paracompact (in particular, metrizable) spaces $X$
such that the free abelian topological group $A(X)$ is a Schwartz group. 

For a subset $U$ of an abelian group $G$ such that $0 \in U$, and a natural
number $n$, we denote $U_{(n)} = \{x \in G : kx \in U: k=1, 2, \dots, n\}$ \cite{Chasco}.
The next formulation is adjusted to Definition \ref{d:Sch}(b) of a Schwartz LCS.

\begin{definition}\cite{Chasco}\label{d:Schwartz_group}
Let $G$ be an abelian topological group. We
say that $G$ is a Schwartz group if for every neighborhood $U$ of zero in $G$
there exists another neighborhood $V$ of zero in $G$ and a sequence $(F_n)$ of
finite subsets of $G$ such that $V \subseteq F_n + U_{(n)}$ for every $n \in \N$.
\end{definition}

Below we collect several basic facts about the topology of $A(X)$ and its relations with the topology  of $L(X)$
 (for the details see \cite{Usp1}).

For a space $X$ we denote by $A_0(X)$ the open subgroup
$$
A_0(X)=\left\{\sum c_i x_i: \sum c_i=0, \ c_i\in \mathbb Z, \ x_i\in X\right\}
$$ 
of $A(X)$.
If $d$ is a continuous pseudometric on $X$, $d$ has a canonical translation-invariant
extension over $A_0(X)$ defined by 
\begin{equation}
\label{eq:A}
\overline d(v,0)=\inf\left\{\sum |c_i| d(x_i,y_i): v=\sum c_i(x_i-y_i), \ c_i\in \mathbb Z, \ x_i, y_i\in X\right\}
\end{equation}
and $\overline d(u,v)=\overline d(u-v,0)$. We denote the seminorm $v\mapsto\overline d(v,0)$
on $A_0(X)$ by the same symbol $\overline d$.
This seminorm is the restriction of the seminorm on $L_0(X)$ 
denoted by the same symbol in the equality (\ref{eq:L1}).
In other words, in the definition above we may replace the condition 
$c_i\in \mathbb Z$ by $c_i\in \mathbb R$ without affecting the result \cite{Usp1}.
The fact that in the equality (\ref{eq:A})
 the minimum over $c_i\in \mathbb Z$ is the same as the minimum
over $c_i\in \mathbb R$ follows from the integrality property for the transportation problem.
Let $a_1,\dots, a_m$ and $b_1,\dots, b_n$ be non-negative integers. Consider the compact convex set
of all $(m\ti n)$-matrices $(c_{ij})$ with real entries $\ge 0$ such that the sum of all entries
in the $i$-th row is $a_i$ and the sum in the $j$-th column is $b_j$ ($1\le i\le m$, $1\le j\le n$).
Then all extreme points of this convex set are matrices with integral entries.

Since the seminorm $\overline d$ on $A_0(X)$ is the restriction of 
a seminorm on the LCS $L_0(X)$, we immediately deduce the following assertion.

\begin{lemma}
\label{l:AL}
Let $d$ be a pseudometric on a set $X$, $\overline d$ the seminorm on $A_0(X)$ given
by the equality~(\ref{eq:A}). For every integer $n$ and $t\in A_0(X)$ we have 
$\overline d(nt)=|n|\overline d(t)$.
\end{lemma}

Denote by $X^{(1)}$ the set of all non-isolated points of $X$.

\begin{theorem}
\label{t:main4}
Let $X$ be a first-countable paracompact space. The following conditions are equivalent: 
\begin{itemize}
  \item[(i)] $A(X)$ is a Schwartz group;
  \item[(ii)] $X$ is a locally compact space such that the set $X^{(1)}$ is $\sigma$-compact.
\end{itemize}
\end{theorem}
\begin{proof} (ii) $\implies$ (i): $X^{(1)}$ is a Lindel\"of space, therefore we can represent $X$ as a disjoint sum $D \oplus Y$
of two clopen subsets $D$ and $Y$,
 where $D$ consists of isolated points of $X$ and $Y$ is locally compact and $\sigma$-compact.
  Further, the group $A(D)$ is discrete and the group $A(Y)$ is Schwartz, since $L(Y)$ is Schwartz, by Theorem \ref{t:loccomp}.
 We conclude that $A(X)$ is a Schwartz group, in view of the known fact that $A(X)$ is isomorphic to the product  $A(D) \times A(Y)$. 

(i) $\implies$ (ii): On the contrary, assume that $X$ is not locally compact. Then, as we have observed earlier in the proof of Theorem \ref{t:main3},
 $X$ contains a closed copy of the metric fan $M$. By Theorem \ref{t:Usp1}, this means that the free abelian topological group $A(X)$ contains an isomorphic copy of 
the free abelian topological group $A(M)$. It remains to prove the following

\begin{fact}\label{l:A(M)}
Let $M$ be the metric fan. Then $A(M)$ is not a Schwartz group.
\end{fact}
\begin{proof}
The fan $M=(\N\ti \N)\cup\{p\}$ can be represented as a countable union of disjoint closed
discrete layers $\{M_k: k \in \N\}$, $M_k=\N\ti\{k\}$, and a single non-isolated point $p$
such that for every choice $x_k \in M_k$ the sequence $(x_k)$ converges to $p$ in $M$.
Consider the natural metric $d$ on $M$ defined as follows: 
the distance $d(x,p) = 1/k$ for every $x \in M_k$;
the distance $d(x,y)= 2/k$ between any two distinct points $x, y \in M_k$;
 and $d(x,y)= 1/k + 1/l$ if $x \in M_k, y \in M_l, k \neq l$. 
%Without loss of generality we may translate the point $p$ to zero of $A(M)$. 
Let $U$ be the unit ball in $A_0(M)$ with respect to the metric $\bar{d}$ on $A_0(M)$.
If $V$ is any neighborhood of zero, we can find a natural $k$ such that all points in the $k$-th layer $M_k$ of the fan are in $p+V$. 
Since $d(x,y) = 2/k$ for any two distinct points $x, y$ of $M_k$, 
we conclude that $\bar{d}(nx,ny)=2n/k$ (by Lemma~\ref{l:AL}).
It follows that for $n > k$ the difference $x-y$ does not belong to $U_{(n)}-U_{(n)}$, so the balls
$x+U_{(n)}$, as $x$ runs over $M_k$, are pairwise disjoint. Therefore, $M_k$ cannot be covered by
a set of the form $F_n + U_{(n)}$, where $F_n$ is finite. The same is true for the translate $M_k - p$
of $M_k$ and moreover for the larger set $V$.
\end{proof}

So, we obtained a contradiction assuming that $X$ is not locally compact. Therefore, $X^{(1)}$ as a closed subset of $X$ is itself locally compact (and paracompact).
Then, as we have observed earlier in the proof of Theorem \ref{t:main3},
being locally compact and paracompact, $X^{(1)}$ can be represented as a disjoint union of clopen (in $X^{(1)}$) $\sigma$-compact subspaces.
Striving for a contradiction, assume that the number of clopen components in this representation is uncountable.
Then we can find an uncountable closed discrete %(in $X^{(1)}$) 
subset $W \subset X^{(1)} \subset X$.
Paracompact spaces are collectionwise normal \cite[Theorem 5.1.18]{Eng}, hence there exists a discrete
family $\{U_w: w\in W\}$ of open sets in $X$ such that $w\in U_w$ for every $w\in W$. 
Every $w\in W$ is the limit of a convergent sequence $S_w\sbs U_w$. 
Put $Y = \bigcup_{w\in W} S_w$. Then $Y$ is closed in $X$ (as the union of a discrete family
of closed sets), and $Y$ is homeomorphic to a disjoint sum of uncountably many convergent sequences.
%Every point of $X^{(1)}$ is the limit of a converging countable sequence consisting 
%of isolated in $X$ points.
% Hence, we can find an uncountable closed discrete (in $X^{(1)}$) subset $W \subset X^{(1)}$ 
%and pairwise disjoint countable sequences $C^{(w)}=(x_n^{(w)}) \subset X \setminus X^{(1)}$ 
%such  that
%$(x_n^{(w)})$ converges to $w$ in $X$, for each $w \in W$. 
%Denote by $S^{(w)}$ the set $C^{(w)} \cup \{w\}$ and put $Y = \bigcup_{w\in W} S^{(w)}$.
%It is easy to see that $Y$ is closed in $X$. 
In order to get a contradiction it now suffices to prove the following

\begin{fact}\label{l:A(Y)}
Let $Y$ be the disjoint sum of uncountably many convergent sequences. 
Then $A(Y)$ is not a Schwartz group.
\end{fact}
\begin{proof}
Let $S=\{1,1/2,1/3,\dots\}\cup\{0\}$ be the standard convergent sequence together with its limit, 
$W$ be an uncountable index set (regarded as a discrete space), and $Y=S\ti W$.
 Equip $Y$ with the natural metric $d$ such that the distance $d$ between $(1/k, w')$ 
and $(1/l, w'')$ is $|1/k - 1/l|$ if $w' = w''$ and $1$ otherwise. Take  
$U=\{t\in A_0(Y):\overline d(t)<1\}$. It suffices to prove that there do not exist
a neighborhood $V$ of
zero in $A_0(Y)$ and finite sets $F_n$ such that $V\sbs F_n+U_{(n)}$ for all $n=1, 2, \dots$.

Assume the contrary: such $V$ and $F_n$'s exist. If $w\in W$ and $\e>0$, let us say that $V$ is
{\it $\e$-fat in the direction} $w$ if $x-y\in V$ for all points $x=(1/k, w)$ and $y=(1/k', w)$
such that $1/k\le \e$ and $1/k'\le \e$. For every $w\in W$ there exists an integer $k\in \N$ such that
$V$ is $1/k$-fat in the direction $w$, so we can choose an infinite set $J\sbs W$ and $k\in \N$ such that
$V$ is $1/k$-fat in the direction $w$ for every $w\in J$. 

For $w\in J$
consider the points $x_w=(1/k, w)$ and $y_w=(1/(k+1), w)$. Note that the distance from $x_w$ to any 
other point in $Y$ is $\ge 1/k(k+1)$. 
%Similarly, the distance from $y_i$ to any other point in $Y$ is $\ge 1/(k+1)(k+2)$. 
From the equality~(\ref{eq:A}) for $\overline d$ we conclude that 
$\overline d(t)\ge 1/k(k+1)$ for every $t\in A_0(Y)$ such that the support of $t$ contains 
$x_w$. 
(The {\em support} $\text{supp\,}t$ of $t=\sum_{y\in Y} n_y y$ 
is the set of all $y\in Y$ such that $n_y\neq0$.) If $n\ge k(k+1)$, we have 
$\overline d(nt)=n\overline d(t)\ge1$ (by Lemma~\ref{l:AL}), so $t\notin U_{(n)}$. Since 
$x_w - y_w\in V\sbs F_n+U_{(n)}$ and the support of any $t\in U_{(n)}$ does not contain $x_w$, 
it follows that there is $t_w\in F_n$ such that $x_w$ belongs to the support of $t_w$.
Since this is true for every $w\in J$, we get a contradiction: the infinite set 
$\{x_w:w\in J\}$ is covered
by finitely many finite sets $\text{supp\,}t$, $t\in F_n$.
\end{proof}

We now can finish the proof of Theorem~\ref{t:main4}:
$X^{(1)}$ can be represented as at most countable union of clopen $\sigma$-compact subspaces, and the proof is complete.
\end{proof}

\begin{corollary}
\label{c:main2} Let $X$ be a metrizable space. The following conditions are equivalent: 
\begin{itemize}
  \item[(i)] $A(X)$ is a Schwartz group;
  \item[(ii)] $X$ can be represented as a disjoint union of two clopen subsets $D$ and $Y$, where $D$ consists of isolated points of $X$ and $Y$ is 
	locally compact and has a countable base.
\end{itemize}
\end{corollary}

The last part of this section is devoted to the following goal: find sufficient conditions on $X$, more general than
those in Theorem \ref{t:loccomp}, which imply that $L(X)$ is a Schwartz LCS (and $A(X)$ is a Schwartz group).

Recall that a Tychonoff space is {\em pseudocompact} if every continuous real-valued function defined on $X$ is bounded.
Let us say that a Tychonoff space $X$ is {\em $QS$-pseudocompact}, if there is a countable sequence of pseudocompact spaces $(X_n)$
such that $X$ can be represented as a quotient of the free sum $\bigoplus_{n\in\N} X_n$.

\begin{theorem}\label{t:QS} 
Let $X$ be a $QS$-pseudocompact space. Then $L(X)$ is a Schwartz LCS.
\end{theorem}
\begin{proof} By the argument used in the paragraph after Theorem~\ref{t:main2}, 
the general case can be reduced to the case when $X$ is pseudocompact. 
In that case $L(X)$ is isomorphic to a linear subspace of $L(\beta X)$, 
where $\beta X$ is the \v{C}ech--Stone compactification of $X$
(see \cite{Tk} or \cite{Usp1}).
Since $L(\beta X)$ is Schwartz, by Theorem \ref{t:lx}, we conclude that $L(X)$ also is Schwartz.
\end{proof}
%Notice that $L(X)$ is a quotient of $L(Y)$, where $Y=\bigoplus_{n\in\N} X_n$.
%So, without loss of generality it suffices to show that $L(X)$ is Schwartz for every pseudocompact space $X$.
%Evidently, a pseudocompact space $X$ is $C$-embedded in its Stone-\v{C}ech compactification $\beta X$, i.e.
%every continuous real-valued function on $X$ has a continuous extension to $\beta X$.
%Observe that 
% every discrete family of nonempty open subsets of a pseudocompact space $X$ is finite (see \cite[Theorem 3.10.22]{Eng}).
%Therefore, by \cite[Theorem 2]{Usp1}, $X$ is $P$-embedded in $\beta X$, i.e. 
%every continuous pseudometric on $X$ extends to a continuous pseudometric on $\beta X$.
%The last property is equivalent to the statement that $L(X)$ is isomorphic to a linear subspace of $L(\beta X)$ (see \cite{Tk} and \cite{Usp1}).
%Since $L(\beta X)$ is Schwartz, by Theorem \ref{t:lx}, we conclude that $L(X)$ also is Schwartz.

\begin{corollary}\label{c:QS}
Let X be a free sum of a discrete space and a $QS$-pseudocompact space. Then $A(X)$ is a Schwartz group.
\end{corollary}

\begin{example}\label{omega_1}
{\em
Let $X$ be an ordered topological space $[0,\omega_1)$ consisting of all countable ordinals.
Then $X$ is a countably compact (hence pseudocompact) first-countable locally compact space.
Further, $X$ is not $\sigma$-compact.
Nevertheless, by Theorem \ref{t:QS}, $L(X)$ is a Schwartz LCS. There is no contradiction with Theorem \ref{t:main3}
because $X$ is not paracompact.
} 
\end{example}

%%%%%%%%%%%%%%%%%%
\section{Open problems}\label{s:problems}
\bigskip
%General questions remained open.
\begin{problem}\label{prob1}
Characterize all Tychonoff spaces $X$ such that $L(X)$ are nuclear / multi-Hilbert.
\end{problem}

\begin{problem}\label{prob2}
Characterize all Tychonoff spaces $X$ such that $L(X)$ are Schwartz / multi-reflexive.
\end{problem}

\begin{problem}\label{prob3}
Does there exist a Tychonoff space $X$ such that $L(X)$ is multi-Hilbert but not nuclear?
\end{problem}

\begin{problem}\label{prob4}
Does there exist a Tychonoff space $X$ such that $L(X)$ is multi-reflexive but not Schwartz?
\end{problem}

A more specific question is the following.

\begin{problem}\label{prob5}
Let $p$ be a point from the the remainder of the Stone-\v{C}ech compactification $\beta \mathbb N \setminus \mathbb N$.
Denote by $X=\mathbb N \cup \{p\}$. Is $L(X)$ nuclear or multi-Hilbert?
\end{problem}

Note that 
every Schwartz space can be isomorphically embedded in some sufficiently high power of the Banach space $c_0$ \cite{Randtke2}.
There are examples of Schwartz spaces which cannot be isomorphically embedded in any power of the Banach space $l^p$ $(1 < p < \infty)$ \cite{Bel1}.
On the other hand, every nuclear space can be isomorphically embedded in some sufficiently high power of every Banach space \cite{Saxon}. 
These facts provide motivation for the following %natural 
questions.

\begin{problem}\label{prob6}
Characterize all Tychonoff spaces $X$ such that $L(X)$ can be isomorphically embedded in some sufficiently high power of the Banach space $c_0$. 
\end{problem}

\begin{problem}\label{prob7}
Characterize all Tychonoff spaces $X$ such that $L(X)$ can be isomorphically embedded in some sufficiently high power of the Banach space $l^p$ $(1 \leq p < \infty)$. 
\end{problem}

%%%%%%%%%%%%%%%%%%%%%%%%%%%%%%%%%%%%%%%%%%
\textbf{Acknowledgements.} We would like to express our gratitude to Vladimir Pestov and Michael Megrelishvili 
for the most stimulating questions and discussions on the topic. We thank Mikhail Ostrovskii for %providing information about 
guiding us towards the paper \cite{Joichi}.

\end{document}